\documentclass[a4paper,12pt]{article}

\newenvironment{proof}{\noindent {\bf Proof:}}{\hfill $\Box$}

\newtheorem{lemma}{Lemma}
\newtheorem{proposition}{Proposition}

\newtheorem{remark}{Remark}

\newtheorem{example}{Example}

\usepackage{booktabs}
\usepackage[T1]{fontenc}
\usepackage{amsmath}
\usepackage{amssymb}
\usepackage{amsfonts}
\usepackage{mathrsfs} 
\usepackage{graphicx}
\usepackage{color}
\usepackage{bm}
\usepackage{hyperref}

\newcommand{\R}{\mathbb{R}} 
\newcommand{\N}{\mathbb{N}} 

\def\M{\mathbf{M}}
\def\b{\boldsymbol{b}}
\def\p{\boldsymbol{p}}
\def\q{\boldsymbol{q}}
\def\v{\boldsymbol{v}}

\begin{document}
\title{An infinite-dimensional Christoffel function and detection of abnormal trajectories\footnote{Didier Henrion is supported by the European Union under the project ROBOPROX (reg. no.
	CZ.02.01.01/00/22 008/0004590).
	J.B. Lasserre is supported by the AI Interdisciplinary Institute ANITI funding through the French program
 ``Investing for the Future PI3A'' under the grant agreement number ANR-19-PI3A-0004. He also affiliated with
	IPAL-CNRS laboratory, Singapore.}
}

\author{Didier Henrion$^{1,2}$, Jean Bernard Lasserre$^{1,3}$}

\footnotetext[1]{CNRS; LAAS; Universit\'e de Toulouse, 7 avenue du colonel Roche, F-31400 Toulouse, France. }
\footnotetext[2]{Faculty of Electrical Engineering, Czech Technical University in Prague, Technick\'a 2, CZ-16626 Prague, Czechia.}
\footnotetext[3]{Toulouse School of Economics (TSE); Toulouse, France. }

\date{June 2024}

\maketitle
\begin{abstract}
We introduce an infinite-dimensional version of the Christoffel function, where now (i) its argument lies in a Hilbert space of functions, and (ii) its associated underlying measure is supported on a compact subset of the Hilbert space. We show that it possesses the same crucial property as its finite-dimensional version to identify the support of the measure (and so to detect outliers). Indeed, the growth of its reciprocal with respect to its degree is at least exponential outside the support of the measure and at most polynomial inside. Moreover, for a fixed degree, its computation mimics that of the finite-dimensional case, but now 
the entries of the moment matrix associated with the measure are moments of moments. To illustrate the potential of this new tool, we consider the following application. Given a data base  of registered reference trajectories, we consider the problem of detecting whether a newly acquired trajectory is abnormal (or out of distribution) with respect to the data base. As in the finite-dimensional case, we use the infinite-dimensional Christoffel function as a score function to detect outliers and abnormal trajectories. A few numerical examples 
are provided to illustrate the theory.
\end{abstract}

\section{Motivation}

Suppose that we are given a set of $N$ trajectories
$\mathscr{D}:=\{(t,g_j(t)): t\in [0,1]\,;\: j\in [N]\,\}\subset [0,1]\times\R$, where $[N]:=\{1,\ldots,N\}$
and $g_j\in L^2([0,1])$ for all $j$.
For example,
each $j\in [N]$ refers to some patient and $\{(t,g_j(t)): t\in [0,1]\}$ is some recorded
data from that patient (e.g. time evolution of blood pressure after injection of some liquid product). 
The set $\mathscr{D}$ can correspond 
to a class of normal patients (e.g. with no pathology), or a class of patients 
whose time evolution of the observed quantity, has a specific characteristic.
Then the problem that we want to address is:

\emph{Given a data base $\mathscr{D}$ of trajectories associated with some class of
patients, detect whether a newly registered trajectory $\{(t,f(t)):t\in [0,1]\}$ of a new patient, 
is abnormal (with respect to this class of patients), or not}. \\

One possible approach is to compute a \emph{score function} of the form
\begin{equation}
\label{intro-norm}
\rho(f)=\min_{g\in\mathscr{D}}\Vert  g-f\Vert_2 \quad\left(\mbox{e.g.}\quad \min_{g\in\mathscr{D}} \int_0^1 (g(t)-f(t))^2\,dt\,\right)^{1/2}\,,
\end{equation}
and decide abnormality of $f$ w.r.t. $\mathscr{D}$, depending on how large is the score $\rho(f)$. In this approach,
computing the score $\rho(f)$ requires to compare $f$ with 
\emph{every} $g\in\mathscr{D}$ for every new trajectory $f$.  Moreover, to be successful, such a test requires the data base $\mathscr{D}$ to be perfect, i.e., with no outlier.

On the other hand, a highly desirable solution is a numerical test obtained via: 

-- (i) some score function $\theta:L^2([0,1])\to\R$, 
$f\mapsto\theta(f)$, 
\emph{computed once and for all},

-- (ii) a threshold $\tau$, such that $\theta(g)\leq \tau$ for all $g\in\mathscr{D}$, so that $f$ is declared abnormal if $\theta(f)>\tau$. If not then the database $\mathscr{D}$ is possibly enriched with the new trajectory
$\{(t,f(t): t\in [0,1]\}$ (considered to be normal), i.e., $\mathscr{D}:= \mathscr{D}\cup\{f\}$. 

In addition:

-- (iii) one should be able to compute the score function $\theta$ efficiently,  and 

-- (iv) update the score function $\theta$
efficiently when additional trajectories are added to (or deleted from) $\mathscr{D}$.\\

\noindent
Motivated by previous works in \cite{neurips,lpp22} where properties of the Christoffel function 
have been proved to be useful for some problems of data analysis (e.g. for outlier detection \cite{louise,neurips}, interpolation \cite{recovery}, optimal transport \cite{mula}) 
we propose an infinite-dimensional version of this approach where now
a point of $\R^n$ is replaced with a trajectory (i.e., a function of time) and the cloud of data points
(a subset of $\R^n$) is now replaced with a database $\mathscr{D}$ of recorded trajectories.

It turns out that there is an appropriate analytical framework, an infinite-dimensional Hilbert space,
in which:

-- An infinite-dimensional analogue of the Christoffel-Darboux (CD) polynomial kernel and Christoffel function (CF) can be defined in relatively simple terms.

-- The dichotomy property of the CF's growth with its degree, depending on whether the CF is evaluated in the support or outside the support of the underlying measure, is also preserved in the infinite-dimensional setting. This property is crucial to detect anomalies
by our numerical scheme, the analogue of the one  proposed in \cite{louise,lpp22} in the finite-dimensional setting. 

That these two essential features are preserved is somehow a little surprising because the infinite-dimensional analogues of polynomials, measures and their support are not trivial extensions of the finite-dimensional case and require some care, even in an appropriate Hilbert-space setting. This setting was initiated in \cite{hikv23,hr24} for questions related to dynamical systems, non-linear partial differential equations and infinite-dimensional optimization. The present work can be seen as an additional contribution to this body of works.

A Matlab prototype code illustrates the developments. It uses the Chebfun package \cite{chebfun}, an object-oriented system for performing numerical linear algebra on continuous functions and operators rather than the usual discrete vectors and matrices.

\begin{remark}
\label{rem:naive}
Warning: In a  naive approach that mimicks the finite-dimensional case, one could:

-- consider the set $\Omega\subset\mathbb{R}^2$ occupied by the trajectories of $\mathscr{D}$ (or its finite version of points in $\mathbb{R}^2$ via samples of trajectories),

-- define a measure $\mu$ on $\Omega$ (e.g. the Lebesgue measure on $\Omega$, or the empirical measure supported on the sampled points of trajectories), its associated CF
$\Lambda^{\mu}_d$ (the reciprocal of a degree-$2d$ bivariate sum-of-squares polynomial),  and a threshold $\delta>0$, 

-- and decide that $f$ is an outlier if 
$\mu(\{t: \Lambda^\mu_d(t,f(t))<\delta\})$ is sufficiently large.

Indeed in such an approach, all trajectories contained in $\Omega$ would be declared inliers, which is obviously not true. The reason why  
this test fails (except when some portion of the trajectory $f$ is indeed outside $\Omega$) is 
that in contrast to \eqref{intro-norm}, the score function
$\Lambda^{\mu}_d(t,f(t))$  applies to \emph{each point} $\{(t,f(t)\}$ of the trajectory $f$, \emph{separately}.
In other words, at every time $t\in [0,1]$,
$\min_{g\in\mathscr{D}} \vert f(t)-g(t)\vert$ can be small even if $\rho(f)$ in \eqref{intro-norm}
is not small.

So  we need
a score function $f\mapsto \theta(f)$ for the entire trajectory $\{(t,f(t)): t\in [0,1]\}$ taken as a whole and not point by point.
\end{remark}

\section{Preliminaries}

Some of the material of this section has been developed in more details in \cite{h24}. It is reproduced for sake of clarity of exposition.
Let $H$ be a separable real Hilbert space, equipped with an inner product $\langle .,.\rangle$ and let $(e_k)_{k=1,2,\ldots}$ be a complete orthonormal system in $H$, with $e_1=1$.  For instance, in the framework of the next section (and also the introduction), a natural choice for $H$ can be :

-- $L^2([-1,1],\phi)$ with $\phi=dt/\pi\sqrt{1-t^2}$ and $e_k$ is the degree $k-1$ Chebyshev polynomial of the first kind (normalized to obtain an orthonormal basis), or

-- $L^2([0,1],\phi)$ with $\phi=dt$ and $e_k$ is the degree $k-1$ (shifted) Legendre polynomial 
(normalized to obtain an orthonormal basis).

It is also possible to consider complex valued and periodic functions. In this case $e_k$ could be a complex exponential $\text{exp}(2\pi{\bf i}k t)$.
In this paper, for ease of exposition, however we restrict our attention  to real non-periodic functions.

For any $n \in \N$, consider the projection mapping
\[\begin{array}{llll}
\pi_n : & H & \to & H \\
& f & \mapsto & \displaystyle\sum_{k=1}^n \langle f,e_k \rangle\, e_k.
\end{array}\]
In particular, note that $f = \lim_{n\to\infty} \pi_n(f) = \sum_{k=1}^{\infty} \langle f,e_k \rangle\, e_k$ (in norm). Also note that
\begin{equation}\label{np}
|\pi_n(f)|^2 \,=\, \sum_{k=1}^{n} \langle f, e_k \rangle^2\quad\mbox{and}\quad
|f|^2= \lim_{n\to\infty} |\pi_n(f)|^2 \,=\, \sum_{k=1}^{\infty} \langle f, e_k \rangle^2.
\end{equation}

\subsection{Polynomials}

Let $c_0(\N)$ denote the set of integer sequences with finitely many non-zero elements, i.e. if $a=(a_1,a_2,\ldots) \in c_0(\N)$ then $\text{card}\:a := \{k : a_k \neq 0\}< \infty$.
Let us define the {\it monomial} of degree $a \in c_0(\N)$ as
\[
f^a := \prod_{k=1,2,\ldots} {\langle f, e_k \rangle}^{a_k}.
\]
This is a product of finitely many powers of linear functionals. Then {\it polynomials} in $H$ are defined as finite linear combinations of monomials $f^a$, that is:
\[
f\mapsto p(f) = \sum_{a \in \text{spt}\:p} p_a f^a\,,\quad \forall f\,\in\,H\,,
\]
where $\text{spt}\:p:=\{a:\,p_a\neq0\}$ is a finite set.
An important difference with the finite-dimensional setting
is that a polynomial $p$ on $H$ has now two notions of degree.

-- Its \emph{algebraic degree} is $d:=\max_{a \in \text{spt}\:p} \sum_k\:a_k$, which
corresponds to the total degree in the finite-dimensional setting, and

-- its  \emph{harmonic degree} is $n:=\max_{a \in \text{spt}\:p} \max_{a_k \neq 0} k$, which 
refers to its finite number of variables.

For instance if $H=L^2([-1,1],\phi)$ with $\phi=dt/\pi\sqrt{1-t^2}$ (the Chebyshev measure) and $e_1(t)=1,e_2(t)=\sqrt{2}t,e_3(t)=\sqrt{2}(2t^2-1),\ldots$, the Chebyshev polynomials of the first kind (normalized to make them an orthonormal basis w.r.t. $\phi$), then 
\begin{equation}
    \label{exemple}
    f\mapsto p(f)\,=\,\displaystyle\sum_{a\in\text{spt}\:p}
p_a\,\prod_{k=1,2,\ldots}\left(\int_{-1}^1 f(t)\,e_k(t)\,dt\right)^{a_k}\,,\quad\forall f\in H\,.\end{equation}

Given $d,n \in \N$, let $P_{d,n}$ denote the finite-dimensional vector space of polynomials of algebraic degree up to $d$ and harmonic degree up to $n$. 
Like $\mathbb{R}[x_1,\ldots,x_n]_d$ in the finite-dimensional setting, its dimension is the binomial coefficient ${n+d\choose n}$. 
Each polynomial $p(.) \in P_{d,n}$ can be identified with its coefficient vector $\p:=(p_a)_{a \in \text{spt}\:p} \in \R^{n+d\choose n}$. 
For example, if $d=4$ and $n=2$, then $P_{4,2}$ has dimension  ${6\choose 2}=15$. The monomial of degree $a=(3,1,0,0, \ldots)$ is $\langle x,e_1\rangle^3\langle x,e_2\rangle$. It belongs to $P_{4,2}$ since it has algebraic degree $4$ and harmonic degree $2$.

\begin{remark}
    A polynomial $f\mapsto p(f)$, with $f\in H$, is 
    a sum of product of integral terms (e.g. as in \eqref{exemple}) and should not be confused with the \emph{function} $t\mapsto p(f(t))$, which is en element of $H$ (as $f(t)$ is) with its own development $\sum_{k=1,2,\ldots} a_k\,e_k$
    in the orthonormal basis $(e_k)_{k=1,2,\ldots}$.  For instance with $H=L^2([-1,1],\phi)$, $\phi$ the Chebyshev measure on $[-1,1]$ and $(e_k)_{k=1,2,\ldots}$ the (normalized) Chebyshev polynomials, if $t\mapsto f(t)=\sqrt{2} t$,
    the function $t\mapsto f(t)^2=2t^2$ is the element $\sqrt{2}e_3/2+e_1$ of $H$. On the other hand with $a=(0,2,0,0,\ldots)$,
    \[f\mapsto p(f):=\,f^a\,=\,\langle f,e_2\rangle^2\,=\,\left(\int_{-1}^1f(t)\,\sqrt{2}tdt/\pi\sqrt{1-t^2}\right)^2\,,\]
    which when evaluated at $t\mapsto f(t)=\sqrt{2}t = e_2 \in H$,
    yields the value 
        \[\left(\int_{-1}^1 2t^2dt/\pi\sqrt{1-t^2}\right)^2\,=\,\Vert e_2\Vert^2\,=\,1\,.\]    
\end{remark}

\subsection{Moments}

With $F$ a compact subset of $H$, let $C(F)$ denote the space of continuous functions on $F$, and let $P(F) \subset C(F)$ denote the space of polynomials on $F$.

A useful characterization of compact sets in separable Hilbert spaces is as follows \cite[item 45, p.346, IV.13.42]{Dunford}.
\begin{proposition}\label{compact}
A closed bounded set $F \subset H$ is compact if and only if for all $\epsilon \in \R$ there exists $n \in \N$ such that $\sup_{f \in F} |f|^2-|\pi_n(f)|^2 < \epsilon^2$.
\end{proposition}

An example of a compact set is the ellipsoid
\[
	\{f \in H : \sum_{k=1,2,\ldots} a_k \langle f,e_k \rangle^2 \leq 1\}
\]	
if the sequence  $(a_k)_{k=1,2,\ldots}$ is strictly positive and strictly increasing, e.g. $a_k = k $. Another example of a compact set  is the Hilbert cube
\[
	\{f \in H : |\langle f,e_k \rangle| \leq \frac{1}{k}, \: k=1,2,\ldots\}
	\]
see \cite[Item 70 p. 350]{Dunford}.

For an introduction to measures on infinite-dimensional spaces (either as countably additive functions on a sigma algebra, or as dual to continuous functions, or both),
the interested reader is referred to e.g. \cite{rf10}.
Given a measure $\mu$ supported on $F$, and given $a \in c_0(\N)$, the {\it moment} of $\mu$ of degree $a$ is the number
\[
\int_F f^a \,d\mu(f)\,.
\]
A measure on $H$ is uniquely determined by its (countably infinite) sequence of moments.
\begin{proposition}
	Let $\mu_1$ and $\mu_2$ be measures on $F$ such that
	\[
	\int_F f^a \,d\mu_1(f) \,=\, \int_F f^a \,d\mu_2(f)
	\]
	for all $a \in c_0(\N)$. Then $\mu_1= \mu_2$.
\end{proposition}
\begin{proof}
First observe that the set  of polynomials on $F$ is an algebra (i.e. the product of two elements of $P(F)$ is an element of $P(F)$) that separates points (i.e. for all $f_1 \neq f_2 \in F$, there is a $p \in P(F)$ such that $p(f_1) \neq p(f_2)$) and that contains constant functions (corresponding to an empty support). From the Stone-Weierstrass Theorem \cite[Section 12.3]{rf10} in (not necessarily finite-dimensional) topological Hausdorff spaces, it follows that
$P(F)$ is dense in $C(F)$.

Now, if all moments of $\mu_1$ and $\mu_2$ coincide, then 
\[
\int_F p(f) \,d\mu_1(f) \,=\, \int_F p(f) \,d\mu_2(f)\,,
\]	
for all polynomials on $F$, and by density, for all continuous functions on $F$.

Since continuous functions are bounded on a compact set $F$, we can use \cite[Prop. 1.5]{d06} to conclude.
\end{proof}

\section{Infinite-dimensional Christoffel-Darboux polynomial}

Let $\mu$ be a given probability measure on a given compact set $F\subset H$.
Given $d,n \in \N$, the finite-dimensional vector space $P_{d,n}$ of polynomials of algebraic degree up to $d$ and harmonic degree up to $n$ is a Hilbert space once equipped with the inner product 
\[\langle p,q \rangle_{\mu}:=\int_F p(f)q(f)\,d\mu(f).\] 
Let $\v_{d,n}(.)$ denote a basis for $P_{d,n}$, of dimension $n+d\choose n$. Any element $p \in P_{d,n}$ 
can be expressed as $p(.) = \p^T \v_{d,n}(.)$ with $\p$ a vector of coefficients.

Let
\begin{equation}
    \label{matrix-mom}
\M_{d,n}(\mu) \,:=\, \int_F \v_{d,n}(f)\,\v_{d,n}(f)^T \,d\mu(f)
\end{equation}
be the \emph{moment matrix} of order $(d,n)$ of measure $\mu$, which is the Gram matrix of the inner products of pairwise entries of vector $\v_{d,n}(.)$. This matrix is positive semi-definite of size $n+d\choose n$.
If it is non-singular, then via its singular value decomposition, it can be written as
\[
\M_{d,n}(\mu) \,=\, Q S Q^T \,=\, \sum_{i=1}^{n+d\choose n} s_{ii}\, \q_i \q^T_i 
\]
where $S=(s_{ij})$ is the diagonal matrix of singular values, $q_i(f):=\q^T_i \v_{d,n}(f)$, and $\q^T_i \M_{d,n}(\mu) \q_j=s_{ii}$ if $i=j$ and $0$ if $i \neq j$.

Assuming that $\M_{d,n}(\mu)$ is non singular, the Christoffel-Darboux (CD) polynomial kernel is then defined by:
\begin{equation}
    \label{cd-kernel}
K^{\mu}_{d,n}(f,g):= \sum_{i=1}^{n+d\choose n} s_{ii}^{-1}\, q_i(f) q_i(g) \,=\, 
\v_{d,n}(f)^T\M_{d,n}(\mu)^{-1}\,\v_{d,n}(g)\,,\quad\forall x,y\,\in\,H\,.
\end{equation}
So the arguments $f,g\in H$ of $K^\mu_{d,n}$ are infinite-dimensional objects. However, for each couple $(d,n)$, $K^\mu_{d,n}$ only sees finitely many of their components.

The CD polynomial  is defined as 
the diagonal of the kernel, that is:
\begin{equation}
    \label{cd-polynomial}
f\mapsto p^{\mu}_{d,n}(f) \,:=\, K^{\mu}_{d,n}(f,f)\,,\quad\forall f\,\in\,H\,.
\end{equation}
Observe that $p^\mu_{d,n}$ is a sum-of-squares polynomial with a Gram matrix which is the inverse of the degree $(d,n)$ moment matrix associated with $\mu$.

\begin{lemma}
The vector space $P_{d,n}$, equipped with $\langle\cdot,\cdot\rangle_\mu$ and the CD kernel $K^\mu_{d,n}$, is a RKHS (reproducible kernel Hilbert space).
\end{lemma}

\begin{proof}
The linear functional $p \mapsto \langle p(.), K^{\mu}_d(.,g)\rangle_{\mu}$ has the reproducing property:
\begin{align*}
\forall p \in P_{d,n}, \quad\langle p(.), K^{\mu}_{d,n}(.,g)\rangle_{\mu} & = \int_F p(f)K^{\mu}_{d,n}(f,g)\,d\mu(f) \\
& = \int_F p(f)\,\v_{d,n}(f)^T \M_{d,n}(\mu)^{-1} \,\v_{d,n}(g)\, d\mu(f)\\
& = \p^T \left(\int_F \v_{d,n}(f)\, \b_{d,n}(x)^T d\mu(f)\right) \M_{d,n}(\mu)^{-1} \,\v_{d,n}(g) \\
& = \p^T \,\v_{d,n}(g) \,=\, p(g)\,,
\end{align*}
and it is continuous:
\[
\forall p \in P_{d,n}, \quad \langle p(.), K^{\mu}_{d,n}(.,g)\rangle_{\mu}^2 \leq \int_F p(f)^2\,d\mu(f) \int_F K^{\mu}_d(f,f)\,d\mu(f).
\]
\end{proof}
\begin{remark}(Degenerate case)
\label{degenerate}
It may happen that $F\subset H$ is contained in a finite-dimensional subspace of $H$. For instance this is the case if there exists $n_*\in\N$ such that 
$\langle f,e_k\rangle=0$ for all $f\in F$ and all $k>n_*$. 
Then $F\subset V$, where $V\subset H$ is the finite-dimensional subspace
associated with the coordinates $\langle f,e_k\rangle$, $k\leq n_*$. In this case the 
moment matrix $\M_{d,n}(\mu)$ in \eqref{matrix-mom} is singular as soon as $n>n_*$. In restricting to 
moments associated with harmonic degree $n\leq n_*$ one is back to the 
classical finite-dimensional studied in e.g. \cite{lpp22}. 

Infinite-dimensional functions $f$ that are specified by finitely many coordinates $\langle f,e_k\rangle$ can be handled as finite-dimensional objects. This is the essence of the Representer Theorem often used in machine learning, see e.g. \cite[Section 5.3]{hr24} in the context of polynomial optimization.
\end{remark}

\section{Infinite-dimensional Christoffel function}

Let $F\subset H$ and $\mu$ be a measure supported on $F$.
The degree $(d,n)$ Christoffel function (CF) associated with $\mu$ is defined by
\begin{equation}\label{cf}
\begin{array}{llll}
\Lambda^{\mu}_{d,n} \: : \: & H & \to & [0,1] \\
& h  & \mapsto &  \displaystyle \min_{p \in P_{d,n}} \int_F p^2(f)\,d\mu(f) \quad\text{s.t.}\quad p(h)\,=\,1\,.
\end{array}
\end{equation}

\begin{lemma}
Assume that $\M_{d,n}(\mu)$ is non singular. Then
for each $h \in H$, the minimum is
\[
\Lambda^{\mu}_{d,n}(h) = \frac{1}{K^{\mu}_{d,n}(h,h)} = \frac{1}{p^{\mu}_{d,n}(h)}\,
\]
and the minimum is achieved at the polynomial
\[
f\mapsto p(f) \,=\, \frac{K^{\mu}_{d,n}(f,h)}{K^{\mu}_{d,n}(h,h)}\,,\quad\forall f\,\in\, H\,.
\]
Finally we also have the alternative formulation
\begin{equation}
    \label{ABC}\Lambda^\mu_{d,n}(h)^{-1}\,=\,\v_{d,n}(h)^T\M_{d,n}(\mu)^{-1}\v_{d,n}(h)\,.
\end{equation}
\end{lemma}

\begin{proof}
Let $h\in H$ be fixed, arbitrary, and let $p\in P_{d,n}$ be a feasible solution of \eqref{cf}. Then
\begin{align*}
1 = p^2(h) & = \left(\int_F K^{\mu}_{d,n}(f,h)p(f)\,d\mu(f)\right)^2 
\quad\mbox{[by the reproducing property]}\\
& \leq \int_F K^{\mu}_{d,n}(f,h)^2\, d\mu(f)
\int_F p^2(f)d\mu(f) 
\quad\mbox{[by Cauchy-Schwarz]}\\
& = K^{\mu}_{d,n}(h,h) \int_F p^2(f)\,d\mu(f)\,,
\end{align*}
and therefore
\[\Lambda^\mu_{d,n}(h) \geq \frac{1}{K^\mu_{d,n}(h,h)}.\]
On the other hand, observe that the polynomial
\[p(.):=\frac{K^{\mu}_{d,n}(.,h)}{K^{\mu}_{d,n}(h,h)} \in P_{d,n}
\]
is admissible for problem \eqref{cf}, i.e. $p(h)=1$, and therefore
\[
\Lambda^{\mu}_{d,n}(h) \leq \int_F \frac{K^{\mu}_{d,n}(f,h)^2}{K^{\mu}_{d,n}(h,h)^2}\,d\mu(f) = \frac{1}{K^{\mu}_{d,n}(h,h)}\,.
\]
\end{proof}

\begin{lemma}\label{average}
The CD polynomial in \eqref{cd-polynomial} has average value
\[
\int_X p^{\mu}_{d,n}(f) \,d\mu(f) \,=\, {n+d\choose n}\,.
\]
\end{lemma}

\begin{proof}
\begin{align*}
\int_F p^{\mu}_{d,n}(f) \,d\mu(f) & = 
\int_F \v_{d,n}(f)^T\M_{d,n}(\mu)^{-1}\,\v_{d,n}(f)\,d\mu(f) \\
& = \text{trace}\left(\M_{d,n}(\mu)^{-1}\cdot\int_F \v_{d,n}(f) \,\v_{d,n}(f)^T\, d\mu(f)\right) \\
& = \text{trace}\:I_{n+d\choose n} \,=\, {n+d\choose n}\,.
\end{align*}
\end{proof}

\begin{remark}
\label{remark-singular}
(i) It may happen that $\M_{d,n}(\mu)$ is singular, e.g. as in the case 
described in Remark \ref{degenerate}, or when $\mu$ is obtained as the empirical measure $\mu_N$ supported on a \emph{finite} sample of $N$ trajectories recorded in a database $\mathscr{D}$ (our motivation discussed in the introduction). In the latter case one,
as one uses the CF associated with $\mu_N$ (and not the underlying $\mu$)
one must be careful. Indeed the couple degree-sample size $\{(d,n),N\}$ cannot be arbitrary 
if the empirical CF $\Lambda^{\mu_N}_{d,n}$ has to inherit the same properties as
$\Lambda^\mu_{d,n}$. For every fixed degree $(d,n)$, a large sample size $N$ is enough.
Indeed by invoking the Strong Law of Large Numbers and continuity of eigenvalues,
$\M_{d,n}(\mu_N)$ can be as close to $\M_{d,n}(\mu)$ as desired, almost surely (with respect to samples) provided that $N$ is sufficiently large \cite[Lemma Theorem 6.2.3]{lpp22}.

(ii) For practical computation, and as in the finite-dimensional case \cite{lpp22}, one uses the regularized version
$\M_{d,n}(\mu)+\varepsilon\,\mathbf{I}$ of the moment matrix, for some $\varepsilon>0$, which ensures that 
it is always non singular. So if $\v_{d,n}(f)\in\mathrm{Ker}(\M_{d,n}(\mu))$ then $\Lambda^\mu_{d,n}(f)=0$ while
its regularized version is very small as desired. This allows to treat numerically in the same manner the singular and non singular cases.

(iii) For matching asymptotic properties 
of $\Lambda^\mu_{d,n}$ and $\Lambda^{\mu_N}_{d,n}$ as $d\wedge n\to\infty$, a more delicate analysis is required to provide a receipe to choose $N$ and $(d,n)$ accordingly, as done in 
\cite[Section 6.3]{lpp22} and \cite{esaim} in the finite-dimensional case.  Notice that for practical purposes (and computational reasons) one is restricted to relatively small degree $(d,n)$ and so a large sample $N$ is enough to ensure that $\Lambda^{\mu_N}_{d,n}$ is a valid proxy for $\Lambda^\mu_{d,n}$. In any case, and for numerical reasons, in practice one replaces the inverse of the moment matrix with 
$(\M_{d,n}(\mu)+\varepsilon\,\mathbf{I}_{{n+d\choose n}})^{-1}$ for small $\varepsilon>0$, which always exists; see for instance \cite{lpp22} in the finite-dimensional case.
\end{remark}

\begin{table}[ht]
\begin{center}
\begin{tabular}{|c|c|}
\hline
\hline
&\\
Standard CF $\Lambda^\mu_d$& New CF $\Lambda^{\mu}_{d,n}$\\
&\\
\hline
&\\
point $x\in \R^{n}$ & point $f\in H$\\
&\\
\hline
point evaluation at $x$ 
& point evaluation at $f$\\
$\delta_{x}(p)\,=\,\langle \p,\v_d(x)\rangle$ &  $\delta_{f}(p)\,=\,\langle \p,\v_{d,n}(f)\rangle$\\
$p\in\R[x]_d$ & $p\in P_{d,n}$\\
&\\
\hline
&\\
measure $\mu$ on $X\subset\R^{n}$ & measure $\mu$ on $F\subset H$ \\
&\\
\hline
&\\
moment matrix & moment matrix \\
$\mathbf{M}_d(\mu)=\displaystyle\int_X \v_d(x)\,\v_d(x)^T\,d\mu(x)$ &
$\mathbf{M}_{d,n}(\mu)\,=\displaystyle\int_F \v_{d,n}(f)\,\v_{d,n}(f)^T\,d\mu(f)$\\
&\\
\hline
&\\
$\Lambda^\mu_d(x)^{-1}=\v_d(x)^T\,\mathbf{M}_d(\mu)^{-1}\,\v_d(x)$ &  
$\Lambda^\mu_{d,n}(f)^{-1}
=\v_{d,n}(f)^T\M_{d,n}(\mu)^{-1}\,\v_{d,n}(f)$
\\
&\\
\hline
\hline
\end{tabular}
\caption{Comparing the standard CF with the new CF.}
\label{tab1}
\end{center}
\end{table}
In Table \ref{tab1} we compare the standard CF in finite dimension with the new CF in infinite dimension. When $x\in\R^n$, the vector $\v_d(x)$ is the vector of 
the usual monomial basis of the vector space $\R[x]_d$ of polynomials of degree up to $d$.

\subsection{Asymptotic properties}
In this section we show that the new infinite-dimensional version $\Lambda^\mu_{d,n}$ on $H$ of the standard
finite-dimensional Christoffel function on Euclidean spaces, also satisfies a distinguished dichotomy property that will prove crucial to detect abnormal trajectories (as it is crucial to detect outliers in the finite-dimensional setting \cite{neurips,lp19,lpp22}).
Let $d\wedge n:=\min(d,n)$.

\begin{lemma}\label{point}
For all $h \in F$, it holds $\lim_{d\wedge n\to\infty} \Lambda^{\mu}_{d,n}(h) = \mu(\{h\})$.
\end{lemma}

\begin{proof}
Let $h \in F$.
First observe that $\Lambda^{\mu}_{d,n}(h)$ is bounded below and non-increasing i.e. $\Lambda^{\mu}_{d'\wedge n'}(h) \leq \Lambda^{\mu}_{d \wedge n}(h)$ whenever $d\wedge n \leq d'\wedge n'$, so $\lim_{d\wedge n\to\infty} \Lambda^{\mu}_{d,n}(h)$ exists.
If $p$ is admissible for problem \eqref{cf}, it holds
\[
\int_F p^2(f)d\mu(f) \leq p^2(h)\mu(\{h\}) = \mu(\{h\})\,,
\] 
so $\lim_{d\wedge n\to\infty} \Lambda^{\mu}_{d,n}(h) \leq \mu(\{h\})$.
Conversely, for given $d,n \in \N$, let
\[
p(f):=(1-|\pi_n(f-h)|^2)^d.
\]
Note that $p(.) \in P_{2d,n}$ and $p(h)=1$ so that $p(.)$ is admissible for problem \eqref{cf} and
\begin{align*}
\Lambda^{\mu}_{2d+1,n}(h) \leq \Lambda^{\mu}_{2d,n}(h) & \leq \int_F p(f)^2 d\mu(f)  \\
& \leq \int_{B(h,d^{-\frac14})}d\mu(f) + \int_{F\setminus B(h,d^{-\frac14})} (1-|\pi_n(f-h)|^2)^{2d} d\mu(f)
\end{align*}
where $B(h,r):=\{f \in F : |f-h| \leq r\}$.
For all $f \in F\setminus B(h,d^{-\frac14})$, using \eqref{np} it holds
\[
|f-h|^2=\sum_{k=1}^n \langle f-h, e_k \rangle^2 + \sum_{k=n+1}^{\infty} \langle f-h, e_k \rangle^2  \geq d^{-\frac12}
\]
and hence
\[
(1-|\pi_n(f-h)|^2)^{2d} = (1-\sum_{k=1}^n \langle f-h, e_k \rangle^2)^{2d} \leq 
(1-d^{-\frac12}+\sum_{k=n+1}^{\infty} \langle f-h, e_k \rangle^2)^{2d}
\]
from which it follows that
\[
\lim_{n\to\infty} \int_{F\setminus B(h,d^{-\frac14})} (1-|\pi_n(f-h)|^2)^{2d} d\mu(f) \leq \lim_{n\to\infty} (1-d^{-\frac12}+\sum_{k=n+1}^{\infty} \langle f-h, e_k \rangle^2)^{2d} = (1-d^{-\frac12})^{2d}.
\]
Combining these asymptotic expressions we get
\[
\lim_{d\wedge n\to\infty} \Lambda^{\mu}_{d,n}(h) \leq \lim_{d\to\infty} \int_{B(h,d^{-\frac14})}d\mu(f) + (1-d^{-\frac12})^{2d} = \mu(\{h\}).
\]
\end{proof}

If $\mu$ is absolutely continuous with respect to e.g. the Gaussian measure restricted to $F\subset H$ \cite{b98,d06}, then 
it follows from Lemmas \ref{average} and \ref{point} that the Christoffel function on $F$ decreases to zero linearly with respect to the dimension of the vector space $P_{d,n}$. Equivalently, on $F$, the CD polynomial $p^\mu_{d,n}$ increases linearly with respect to the dimension. This is in sharp contrast with its exponential growth outside $F$, captured by the following result.

\begin{lemma}\label{outside}
Let $d,n \in \N$. For all $h \in H$ such that $\min_{f \in F}|\pi_n(f-h)| \geq \delta > 0$ it holds
\[
p^{\mu}_{d,n}(h) \geq 2^{\frac{\delta}{\delta+\text{diam}\:F}d-3}\,,
\]
where $\text{diam}\:F:=\max_{f_1,f_2 \in F} |f_1-f_2|$ is finite since $F$ is compact.
\end{lemma}

\begin{proof}
Let $\delta \in (0,1)$ and let
\[
q(f):=\frac{T_d(1+\delta^2-|\pi_n(f)|^2)}{T_d(1+\delta^2)}
\]
where $T_d$ is the univariate Chebyshev polynomial of the first kind of degree $d \in \N$. This polynomial of $f \in H$ is such that 
\begin{itemize}
\item $q(0)=1$,
\item $|q(f)| \leq 1$ whenever $|\pi_n(f)| \leq 1$,
\item $|q(f)| \leq 2^{1-\delta d}$ whenever $0 < \delta \leq |\pi_n(f)| \leq 1$,
\end{itemize}
see \cite[Lemma 6.3]{lp19}. 
Now let
\[p(f):=q\left(\frac{f-h}{\delta+\text{diam}\:F}\right),
\quad \bar{\delta} := \frac{\delta}{\delta+\text{diam}\:F}\,,
\]
If $h\in H$ and $\min_{f \in F}|\pi_n(f-h)| \geq \delta > 0$, then \[0 < \bar{\delta} \leq \left|\pi_n\left(\frac{f-h}{\delta+\text{diam}\:F}\right)\right| \leq 1\,,\]
for all $f \in F$. Note that $p(.) \in P_{2d,n}$ and $p(h)=1$ so that $p(.)$ is admissible in problem \eqref{cf} and hence
\[
\Lambda^{\mu}_{2d,n}(h) \leq \int_F p(f)^2 d\mu(f)
\leq \int_F (2^{1-\bar{\delta}d})^2 d\mu(f) = 2^{2-2\bar{\delta}d} \leq 2^{3-2\bar{\delta}d}.
\]
Also $\Lambda^{\mu}_{2d+1,n}(h) \leq \Lambda^{\mu}_{2d,n}(h) \leq  2^{3-2\bar{\delta}d}$ and since $\bar{\delta} < 1$, it holds $\Lambda^{\mu}_{2d+1,n}(h) \leq 2^{3-\bar{\delta}(2d+1)}$, from which we conclude that $\Lambda^{\mu}_{d,n}(h) \leq 2^{3-\bar{\delta}d}$.

\end{proof}

\subsection{Updating the empirical CF when new data are available}

In the case where $\mu$ is an empirical measure supported supported on the union of trajectories $g\in\mathscr{D}$
(see our motivation in the introduction)  and if an additional trajectory $g_0$ is added to $\mathscr{D}$ (that is,  the support becomes $\mathscr{D}\cup\{g_0\}$)
then as in the finite-dimensional case, the CF $\Lambda^\mu_{d,n}$ is easily updated.
This is because the inverse of the new moment matrix 
is a rank-one update of the initial one, and so 
is easily obtained the Sherman-Morrison-Woodbury formula.

More precisely, let the new trajectory be the element $g_0\in H$ with representation
\[g_0\,=\,\sum_{k=1,2,\ldots} \langle g_0,e_k)\,e_k\,.\]
Then if $\#\mathscr{D}=N$ and $\mu_{N+1}$ is the new empirical measure supported on 
$\mathscr{D}':=\mathscr{D}\cup\{g_0\}$, the inverse of its associated moment matrix satisfies
\begin{eqnarray*}
\left((N+1)\,\M_{d,n}(\mu_{N+1})\right)^{-1}
&=&\left(N\,\M_{d,n}(\mu_{N})+\underbrace{\v_{d,n}(g_0)\v_{d,n}(g_0)^T}_{\mbox{rank-one}}\right)^{-1}\\
&=&\frac{1}{N}(\M_{d,n}(\mu_N)^{-1}-\frac{1}{N^2}\frac{\M_{d,n}(\mu_N)^{-1}\v_{d,n}(g_0)\v_{d,n}(g_0)^T\M_{d,n}(\mu_N)^{-1}}{1+\frac{1}{N}\v_{d,n}(g_0)\M_{d,n}(\mu_N)^{-1}\v_{d,n}(g_0)}\,,
\end{eqnarray*}
and it can be computed efficiently. Indeed note that 
in this  rank-one update, the corrective term only requires computing $\M^{-1}_{d,n}(\mu_N)\v_{d,n}(g_0)$ (one matrix-vector multiplication).

In particular one also obtains the formula:
\[\frac{\Lambda^{\mu_{N+1}}_{d,n}(g)^{-1}}{N+1}\,=\,\frac{\Lambda^{\mu_{N}}_{d,n}(g)^{-1}}{N}-
\frac{\frac{1}{N^2}K^{\mu_N}_{d,n}(g,g_0)^2}{1+\frac{1}{N}\Lambda^{\mu^N}_{d,n}(g_0)^{-1}}\,,\quad\forall g\in H\,.\]

\section{Numerical illustration}

We designed a prototype Matlab code
\begin{center}\url{http://homepages.laas.fr/henrion/software/hilbertcd/chebcd.zip}\end{center}
using the {\tt chebfun} package to manipulate Chebyshev polynomials \cite{chebfun}. Our Hilbert space $H$ is $W^{1,2}([-1,1])$, the Sobolev space of real valued functions on the interval $[-1,1]$ whose weak derivatives are square integrable. Our orthonormal system consists of scaled Chebyshev polynomials of the first kind $e_1=1$, $e_k := \sqrt{2}T_{k-1}$, $k=2,3,\ldots$ 
with $T_1(t):=\sqrt{2}\,t, T_2(t)=\sqrt{2}(2t^2-1), \ldots$ and the scalar product 
\[\langle f, e_k \rangle = \int_{-1}^1 \frac{f(t)T_{k-1}(t)dt}{\pi\sqrt{1-t^2}}\,,\quad\forall f\in\, H\,.\] 
Hence every function $f \in H$ can be expressed as a Chebyshev series $f = \sum_{k=1,2,\ldots} c_k e_k$ where $(c_k = \langle f,e_k \rangle)_{k=1,2,\ldots}\subset \R$ are its associated Chebyshev coefficients.

\begin{example}
\label{ex1}
In a first illustrative example:

-- We fix the algebraic degree $d \in \N$ and harmonic degree $n \in \N$, 

-- we choose a nominal function $g_0 \in H$, and 

-- generate an \emph{empirical}  moment matrix $\M_{d,n}(\mu_N)$ obtained from randomly generated sample functions $g_i$, $i=1,\ldots,N$ around $g_0$:
\[
\M_{d,n}(\mu_N) = \frac{1}{N}\sum_{i=1}^N \v_{d,n}(g_i)\,\v_{d,n}(g_i)^T\,,
\]
where $N$ is much larger than the dimension ${n+d\choose n}$ of the vector space $P_{d,n}$ of polynomials of algebraic degree up to $d$ and harmonic degree up to $n$. Basis $\v_{d,n}(.)$ consists of monomials.

For example, if $d=n=4$ the moment matrix has size ${6\choose 3}=70$. We choose a nominal function $g_0 = \frac{1}{3}(T_1+T_2+T_3)$ and the $N=10^3$ samples $g_i=g_0 + \epsilon_1 T_1 +  \epsilon_2 T_2 +  \epsilon_3 T_3$, $i=1,\ldots,N$ are generated for $\epsilon \in \R^3$ sampled uniformly in the Euclidean ball of radius $1/10$.

Now if we generate a function randomly with the same distribution, the CD polynomial evaluated at this function should have a value around $70$, consistently with Lemma \ref{average}. 

In contrast, if we generate an outlier function using a larger (say by an order of magnitude) perturbation around the nominal function, the CD polynomial evaluated at this function should be much larger, consistently with Lemma \ref{average}. 

\begin{figure}
\begin{center}
\includegraphics[width=0.8\textwidth]{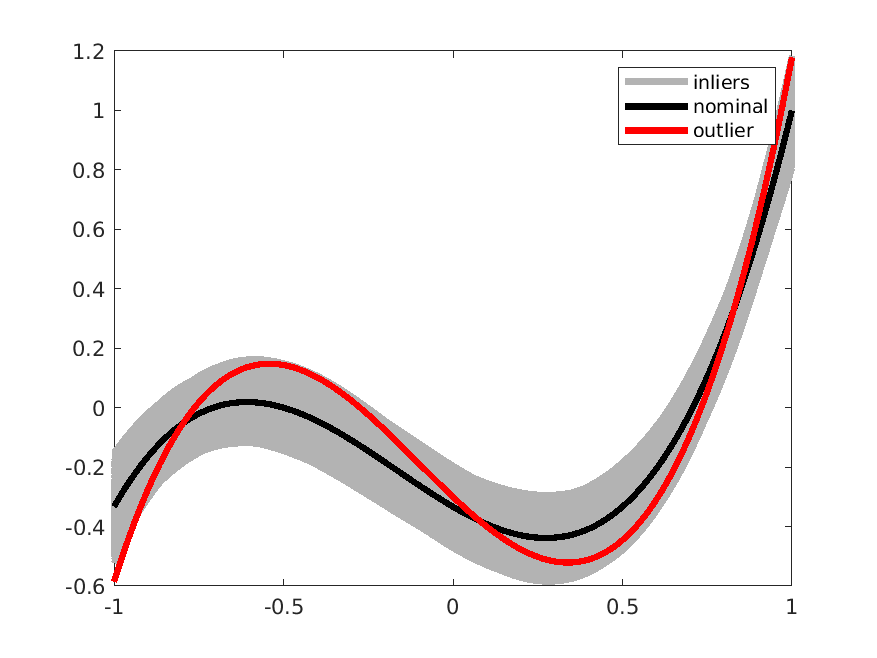}
\caption{Inlier samples (gray), nominal function (black) and outlier function (red).\label{fig:chebinout}}
\end{center}
\end{figure}

\begin{figure}
	\begin{center}
		\includegraphics[width=0.8\textwidth]{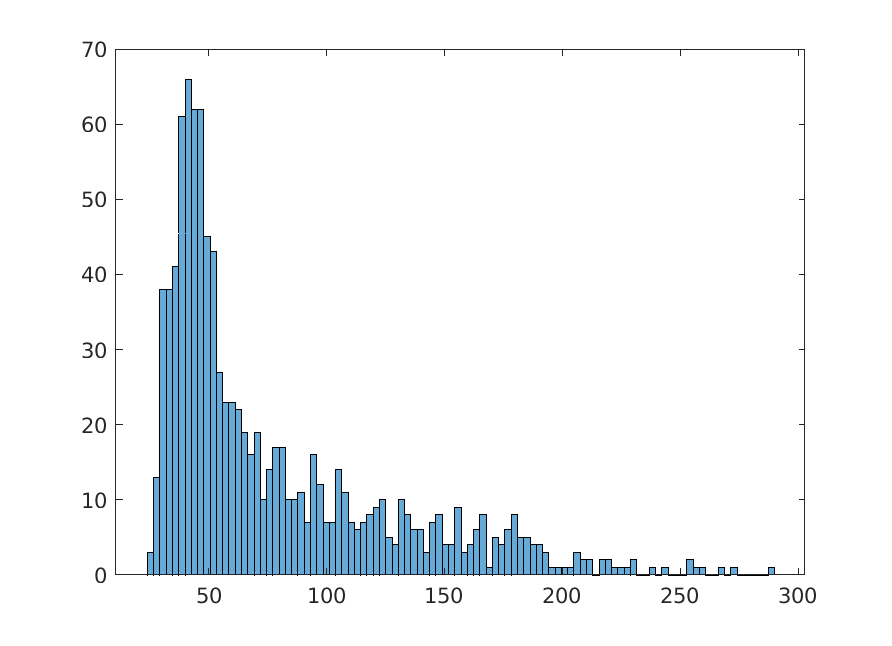}
		\caption{Histogram of values of the CD polynomial evaluated at $10^3$ inlier samples.\label{fig:chebhisto}}
	\end{center}
\end{figure}

On Figure \ref{fig:chebinout} we represent in gray a large number of inliner functions $f_i$, $i=1,2,\ldots,N$ sampled from the same distribution used to generate the moment matrix. In black we represent the nominal function $g_0$.  In red we represent an outlier function $f$, at which the Christoffel-Darboux polynomial takes the large value $3904$, consistently with Lemma \ref{average}.  On Figure \ref{fig:chebhisto} we represent the histogram of values of the CD polynomial at the samples. The average value is $75$, consistently with Lemma \ref{average}. 

This example also showcases why the naive approach alluded to in Remark \ref{rem:naive} fails. 
Indeed observe that the graph of the outlier in red in Figure \ref{fig:chebinout} is contained in the geometrical support of the inliers trajectories in grey  (seen as  points in $\R^2$ and not to be confused with the support of $\mu$ which is an infinite-dimensional object); however it is an \textit{outlier}. 

\begin{figure}
\begin{center}
\includegraphics[width=0.8\textwidth]{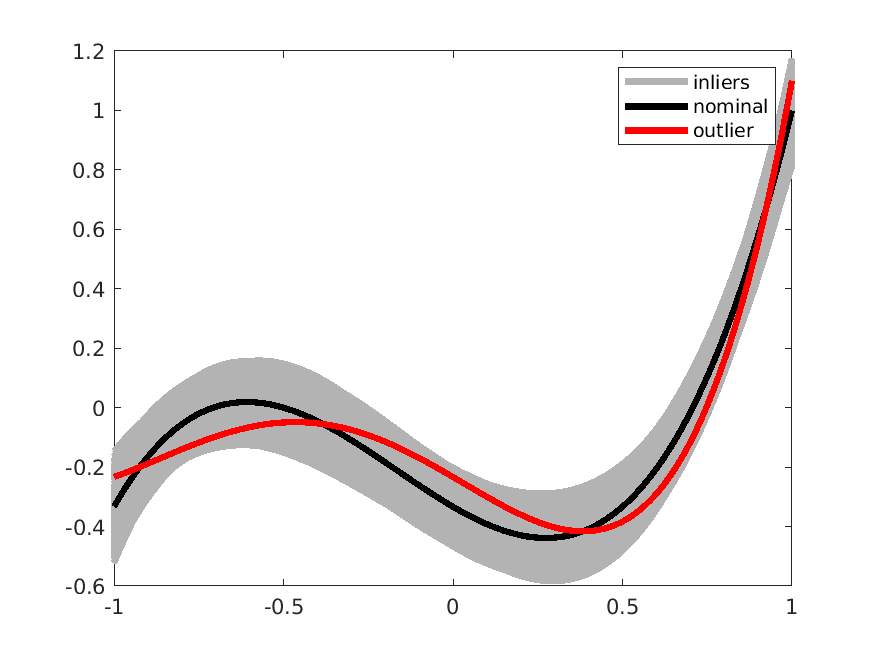}
\caption{Inlier samples (gray), nominal function (black) and outlier function (red).\label{fig:chebout}}
\end{center}
\end{figure}
\end{example}
\begin{example}
We next consider an educational example to illustrate how 
a trajectory $f$ obtained from the nominal trajectory $g_0$ of harmonic degree $n=4$
in Example \ref{ex1}, by adding a perturbation that contains an harmonic not present in 
$g_0$ (e.g. $\langle f,e_{n+1}\rangle\neq 0$) is immediately detected as an outlier. Indeed
consider the empirical moment matrix generated from the same uniform measure as above, but this time fixing the algebraic degree $d=1$ and harmonic degree $n=5$. The moment matrix reads:
\[
\M_{d,n}(\mu_N) = \left(\begin{array}{rrrrrr} 
1 & -0.0024 & 0.3315 & 0.3326  & 0.3352 &    0 \\
-0.0024  &  0.0016 &  -0.0007  & -0.0008  & -0.0008 &        0 \\
    0.3315  & -0.0007  &  0.1116 &   0.1102 &   0.1111      &   0 \\
    0.3326 &  -0.0008 &   0.1102  &  0.1123 &   0.1115   &     0 \\
    0.3352 &  -0.0008 &  0.1111 &  0.1115   & 0.1141   &      0 \\
         0   &      0   &      0    &     0    &     0    &     0
         \end{array}\right)
\]
and its last row and column are identically zero since the samples have harmonic degree $4$ only.
The outlier $g_0+\epsilon T_4$ has harmonic degree $5$, and the value of the CF at this function is zero (and hence the CD polynomial has infinite value) since the moment matrix is not invertible. Yet, on Figure \ref{fig:chebout} we see that the outlier lies in the geometrical support of the inlier trajectories in grey. 
As mentioned in Remark \ref{remark-singular}(ii), in
practice one uses the regularized (and always non singular) version $\M_{d,n}(\mu)+\varepsilon\,\mathbf{I}$ of $\M_{d,n}(\mu)$, in which case the value of the CD polynomial is very large, as desired to detect outliers.

On the other hand, this detection is effective no matter how large is the 
magnitude of the harmonic component $\langle f,e_5\rangle$ whereas the trajectory $f$ is almost indistinguishable graphically from $g_0$ when this component is small, which  may not be 
desirable in a practical situation.
\end{example}

\section{Conclusion}
We have introduced an infinite-dimensional  Christoffel-Darboux polynomial kernel and its associated Christoffel function  on a separable Hilbert space (i.e. whose arguments are infinite-dimensional). It turns out that dichotomy property of the finite-dimensional kernel growth with respect to its degree  (in or outside the support of the underlying measure) is also valid for this infinite-dimensional version. This allows to treat some problems in data analysis (e.g. outlier detection) where the data now consist of a set of reference trajectories (e.g. functions of time in $L^2([0,1])$) instead of 
a set of points in $\R^n$.


\begin{thebibliography}{XX}
	
\bibitem{b98}
V. I. Bogachev. Gaussian measures. AMS, 1998

\bibitem{d06}
G. da Prato. An Introduction to Infinite-Dimensional Analysis. Springer, 2006.

\bibitem{chebfun}
T. A. Driscoll, N. Hale, and L. N. Trefethen (Editors) Chebfun Guide, Pafnuty Publications, Oxford, 2014.

\bibitem{louise}
K. Ducharlet, L. Trav\'e-Massuy\`es, J. B. Lasserre, M.-V. Le Lann,  Y. Miloudi.   Leveraging the Christoffel function for outlier detection in data streams. hal-03562614. To appear in Int. J. Data Sci. Analytics, 2024.

\bibitem{Dunford}
N. Dunford, J. T. Schwartz.  {Linear operators. Part I: general theory},
Interscience Publishers, 1958.

\bibitem{d22}
S. \DJ{}urasinovi\'c. The Christoffel function for supervised learning: theory and practice. MSc Thesis, Univ. Toulouse Capitole, hal-03768886, Sept. 2022.

\bibitem{h24}
D. Henrion.
Infinite-dimensional Christoffel-Darboux polynomial kernels on Hilbert spaces. hal-04628265, June 2024.

\bibitem{hikv23}
D. Henrion, M. Infusino, S. Kuhlmann, V. Vinnikov. Infinite-dimensional moment-SOS hierarchy for nonlinear partial differential equations. hal-04117218. Presented at the SIAM Conference on Optimization, Seattle, USA, May 2023.

\bibitem{hr24}
D. Henrion, A. Rudi. Solving moment and polynomial optimization problems on Sobolev spaces. hal-04393205. To be presented at the MTNS Symposium, Cambridge, UK, August 2024.

\bibitem{recovery}
D. Henrion, J.B. Lasserre.  Graph recovery from incomplete moment information. Constructive Approximation 56:165--187, 2022.

 \bibitem{neurips}
 J.B. Lasserre,  E. Pauwels. Sorting out typicality via the inverse moment matrix SOS polynomial, pp. 190—198 in Advances in Neural Information Processing Systems 29 (NIPS 2016, Barcelona), D. D. Lee, M. Sugiyama,  U. V. Luxburg, I. Guyon, R. Garnett (Editors), Curran Associates, Inc. 2016.
 
\bibitem{lp19}
J. B. Lasserre, E. Pauwels. The empirical Christoffel function with applications in data analysis. Adv. Comput. Math. 45(3):1439-1468, 2019.

\bibitem{lpp22}
J. B. Lasserre, E. Pauwels, M. Putinar.
The Christoffel-Darboux kernel for data analysis.  Cambridge Univ. Press, 2022.

\bibitem{mula}
O. Mula, A. Nouy. Moment-SoS methods for optimal transport. hal-03919946. To appear in Numer. Math. 2024.

\bibitem{p20}
E. Pauwels. Online lectures on Christoffel-Darboux kernels. Marie-Curie Network POEMA (Polynomial Optimization, Efficiency through Moments and Algebra). 17 June, 24 June and 1 July 2020.

\bibitem{rf10}
H. L. Royden, P. M. Fitzpatrick. Real analysis. 4th edition, Prentice Hall, 2010.

\bibitem{esaim}
M. T. Vu, F. Bachoc, E. Pauwels. Rate of convergence for geometric inference based on the empirical Christoffel function, ESAIM Prob. Stat. 26:171--207, 2022.


\end{thebibliography}
\end{document}